\documentclass[11pt]{amsart}

\usepackage{url}
\usepackage{hyperref}

\usepackage[all,cmtip]{xy}
\usepackage{graphicx}
\usepackage{amssymb,amsmath,amscd,graphicx,
latexsym,amsthm}
\usepackage{amssymb,latexsym,amsmath,amscd,graphicx}
\usepackage[mathscr]{eucal}
 \input xy

\usepackage[letterpaper, centering]{geometry}

 \xyoption{all}

\def\Z{{\mathbb Z}}
 
\def\Q{{\mathbb Q}}
                  
\def\P{{\mathbb P}}
\def\K{{\mathbb K}}

\def\OO{{\mathcal O}}

\def\FF{\mathcal{F}}

\def\GG{\mathcal{G}}

\def\Pic{{Pic}}
\def\Pic0{\mathrm{Pic}^0}

\def\max{\mathrm{max}}

\def\deg{\mathrm{deg}}
\def\Char{\mathrm{char}}

\def\Ker{\mathrm{Ker}}

\theoremstyle{plain}

\newtheorem*{introcorollary}{Corollary}
\newtheorem*{introconjecture}{Conjecture}
\newtheorem{theorem}{Theorem}[section]
\newtheorem{theoremalpha}{Theorem}
\newtheorem{corollaryalpha}[theoremalpha]{Corollary}

\newtheorem{proposition/example}[theorem]{Proposition/Example}
\newtheorem{definition/theorem}[theorem]{Definition/Theorem}
\newtheorem{proposition}[theorem]{Proposition}

\newtheorem{corollary}[theorem]{Corollary}
\newtheorem{lemma}[theorem]{Lemma}

\theoremstyle{definition}
\newtheorem{definition}[theorem]{Definition}

\newtheorem{remark}[theorem]{Remark}
\newtheorem{example}[theorem]{Example}

\newtheorem{conjecture/question}[theorem]{Conjecture/Question}

\newtheorem{remark/definition}[theorem]{Remark/Definition}
\newtheorem{notation/assumptions}[theorem]{Assumptions/Notation}

\numberwithin{equation}{section}

 \theoremstyle{remark}

\begin{document}  

\title{On Mukai's conjecture for hyperelliptic varieties}

 \author{Federico Caucci}
\address{Sapienza Universit\`a di Roma, Dipartimento di Scienze di Base e Applicate per l'Ingegneria, Via Antonio Scarpa 16, 00161 Roma, Italy}
 \email{federico.caucci@uniroma1.it}

\maketitle

\setlength{\parskip}{.1 in}

\begin{abstract} 
We prove some general  results on syzygies of smooth projective varieties with numerically trivial canonical line bundle. This allows to
 confirm  several cases of Mukai's syzygies conjecture for 
 finite quotients of  abelian varieties 
 in any dimension, and  in positive characteristic.     
 \end{abstract}

\section{Introduction}
  
When studying syzygies of projective varieties, a  difficult  
conjecture attributed to Mukai has been of pivotal importance. It says the following:
\begin{introconjecture}
Let $X$ be a smooth projective variety, $k \geq 0$ be an integer, and $L$ be an ample line bundle on $X$. If
\[
m \geq \dim X + 2 + k,
\]
then $K_X + mL$ satisfies the property $(N_k)$.
\end{introconjecture}
We refer the reader to \S \ref{S1} for the definition and the geometric meaning of property $(N_k)$. We recall that, by definition, $(N_k)$ implies $(N_{k'})$ for all $0 \leq k' \leq k$. 

Mukai's  conjecture 
 should be seen as a strong generalization of Fujita's conjecture \cite{fuconj}, and it is  very much open,  up to now. 
 Apparently, Mukai was motivated by Green's theorem on higher syzygies for curves \cite{green}. 
  Even the $k=0$ case for surfaces rests unknown in general, while, for instance, Fujita's conjecture  holds true in this case,  by Reider's theorem \cite{reider}. 
On the other hand, some special cases -- with actually better bounds -- were proved. When $S$ is an abelian or a K3 surface, it is known that
$mL$ has the property $(N_k)$ as soon as $m \geq k+3$ (the former  holds true for abelian varieties of arbitrary dimension by \cite{pa1}, while the latter is \cite[Corollary 4.7]{akl}). 
If $S$ is a complex minimal surface of Kodaira dimension zero,
one of the main results of Gallego and Purnaprajna
\cite{gapu} says that    $K_S + mL$ has the property $(N_k)$, assuming $m \geq \max\{4, 2k+2\}$. 

Our first result is a  proof of this  last fact for
 bielliptic surfaces, which works as well in positive characteristic: 
\begin{theoremalpha}\label{thmM}
Let $S$ be a bielliptic surface defined over an algebraically closed field $\K$ of characteristic $\neq 2, 3$. Let
$L$ be an ample line bundle on $S$, and 
 $m, k$  be non-negative integers such that 
$$m \geq \max \{4, 2k+2\}\, .$$
Then,  $K_S + m L$ satisfies the property $(N_k)$.  
\end{theoremalpha}

\begin{introcorollary}
Mukai's conjecture for a  bielliptic surface as above,
holds true when $k \leq 2$.
\end{introcorollary}
\noindent As the reader will notice, 
  the  argument   also proves that, under the same numerical assumption, $m L$  satisfies the property $(N_k)$ as well. Moreover, it properly generalizes to higher dimensions.

Let $A$ be an abelian variety, defined over an algebraically closed field, and 
 $G$ be a finite group  acting freely and not only by translations on $A$.  The quotient variety  $$X := A /G$$ is a \emph{hyperelliptic variety}, which is a higher dimensional version of a bielliptic surface.
Hyperelliptic varieties, or, 
 more generally, non-necessarily smooth, 
finite quotients of abelian varieties, have  attracted attention from different perspectives along the years. Over the complex numbers, 
smooth finite quotients are
 classified in dimension $\leq 2$ by \cite{dim2.1, dim2.2, sh}, and partially in dimension $3$, where we have a complete classification of  the hyperelliptic threefolds (see \cite{dim3}, and the references therein). Indeed, recently there has been a renewed interest in such varieties: see, for instance,  \cite{na, clhoko, grkepe, luta, casurvey, catanese} and the references therein, and especially the 
more recent works of Catanese \cite{catsns, catii}.
First steps towards a classification of hyperelliptic $4$-folds have been    
obtained in \cite{dem}.  
We   refer the reader to  \cite[\S 7]{ueno} and \cite{serrano, ogsa,   
kola,  ala}  for other interesting results on, and   non-trivial examples of,  finite quotients of abelian varieties.  
We also like to mention that hyperelliptic varieties provide interesting examples from the viewpoint of derived categories  (see \cite[\S 4]{kaok} and \cite{bedenu}).

In the paper \cite{chiIy}  the following result is claimed.  
Given an ample line bundle $L$ on a complex hyperelliptic variety $X$, \cite[Theorem 1.3]{chiIy} states that $m L$ has the property $(N_k)$ if $m \geq k+3$.    
Unfortunately, the proof in \cite{chiIy} contains a gap.\footnote{Indeed, the proof of \cite[Lemma 6.3]{chiIy} -- upon which \cite[Theorem 1.3]{chiIy} crucially builds on -- is incorrect as it is stated there that the sub-bundle $\pi^*M_{mL}$
 of the kernel bundle $M_{\pi^*mL}$ on $A$ is a direct summand, i.e., the inclusion $\pi^*M_{mL} \subseteq M_{\pi^*mL}$  splits,
where $\pi \colon A \to X$ is the quotient map, $M_{mL}$  denotes the kernel of the evaluation morphism of global sections $H^0(X, mL) \otimes \OO_X \xrightarrow{\mathrm{ev}} mL$ on $X$, and similarly for
$M_{\pi^*mL}$ on $A$.  
 However, this cannot be the case as, by definition, $M_{\pi^*mL}$ has no non-trivial global sections, while 
 the cokernel of the inclusion $\pi^*M_{mL} \subseteq M_{\pi^*mL}$ is a certain trivial vector bundle $\bigoplus \OO_A$ which is, in general, non-zero. Therefore, the above splitting, if true,    would give that $\bigoplus \OO_A$ is contained in $M_{\pi^*mL}$, and hence $$0 < h^0(A, \bigoplus \OO_A) \leq h^0(A, M_{\pi^*mL}) = 0.$$  
}  
The present paper originated as an attempt to prove (or even disprove) 
such statement, which should be now considered as an open problem. 
In the end, things seem to be  more complicated than one thought, and  
although we have not been able  to  solve   
this problem, we prove some weaker results in arbitrary dimension (see  Theorems \ref{thmC'} and \ref{thmC}  below). 

Let us first state the  second main result of the present paper.  
Recently,  Lacini-Purnaprajna \cite{bala} proved a general theorem on syzygies of smooth complex projective varieties, which answers affirmatively a question
 of Ein-Lazarsfeld \cite[\S 4]{el}: let $Y$ be a complex projective variety of dimension $d$, and let $P$ be an ample and \emph{globally generated} line bundle on $Y$, then $K_Y + mP$ has the property $(N_k)$ if $m \geq d +   1 + k$. 

Assuming $K_Y \equiv 0$, we may improve \cite{bala} via  a fairly elementary argument that only uses Castelnuovo-Mumford regularity, and Kodaira's vanishing (see also Remark \ref{NEF} below): 
\begin{theoremalpha}[$\Char(\K) = 0$]\label{thmY}  
Let $Y$ be a smooth projective variety of dimension $d$, defined over an algebraically closed field $\K$ of characteristic $0$, such that $K_Y$ is numerically trivial. Let $P$ be an ample and globally generated line bundle on $Y$, such that $N+P$ is globally generated for a 
   numerically trivial line bundle $N$ on $Y$. Then, $N + mP$ has the property $(N_k)$, if
\[
m \geq \max\{d+1, k+1\}\, .
\]
\end{theoremalpha}

Since for a hyperelliptic variety $X = A/G$ (defined over $\K$ with $\Char(\K)$ not dividing the cardinality of the acting group $G$) the tensor product of two  ample line bundles on $X$ is globally generated, 
 from Theorem \ref{thmY} and its proof we get:
\begin{theoremalpha}\label{thmC'}
Let $X$ be a hyperelliptic variety of dimension $d$, defined over an algebraically closed field $\K$.
Assume that
 $\Char(\K)$ does not divide $|G|$. 
Let $L$ be an ample line bundle on $X$. Then,  $N + mL$ has the property $(N_k)$ as soon as $$m \geq \max \{2d+1, 2k+2\}\, ,$$ where   $N$ is any numerically trivial line bundle on $X$.
\end{theoremalpha}
\noindent When the acting group $G$  is \emph{commutative},  
 this bound can be  improved, in some cases, as follows:
\begin{theoremalpha}\label{thmC}
Let $X$ be a hyperelliptic variety of dimension $d$ as above.
Assume that $G$ is  a commutative group,  and  that  the Picard variety $(\Pic0 X)_{red}$ of $X$
 is non-trivial. 
Let $L$ be an ample line bundle on $X$. Then,  $N + mL$ with $m \geq  2d$ has the property $(N_1)$, where   $N \equiv 0$  on $X$.
Moreover, it has the property $(N_k)$ for all $k$ such that $2d > 2k+2$.
\end{theoremalpha}
\begin{introcorollary}
Mukai's conjecture holds true for a hyperelliptic variety $X$ as in Theorem \ref{thmC'} when $\dim X -1 \leq k \leq \dim X$, and 
for a hyperelliptic variety $X$ as in Theorem \ref{thmC} when $k = \dim X -2$.
\end{introcorollary}

If $G$ be a finite \emph{cyclic} group acting freely and not only by translations on a complex abelian variety $A$ of dimension $d$,  the quotient $X$ is a so-called \emph{Bagnera-de Franchis variety} (\cite{bacafr, casurv2015}). 
For such a variety, one always has that the dimension of its Picard variety is $= h^1(X, \OO_X) > 0$,\footnote{See  \cite[\S 3]{lange}, and especially Lemma 3.3 of \emph{op.\! cit.}. For an arbitrary group $G$,  the dimension of the Picard variety of $X = A/G$ is not necessarily $\neq 0$ (see, e.g.,  \cite[Example 43]{bedenu} for a $3$-dimensional example  in  characteristic $0$, with $G$ commutative).} hence we get at once:
\begin{corollaryalpha}
Mukai's conjecture holds true for a complex Bagnera-de Franchis variety $X$, when  $\dim X -2 \leq k \leq \dim X$. 
\end{corollaryalpha}
\noindent A complex bielliptic surface is a Bagnera-de Franchis variety of dimension $2$.\footnote{The canonical cover of a bielliptic surface $S$ is an abelian surface $A$, with a free action of $\Z_n$ on it such that $S \simeq A/\Z_n$. Here, $n$ is the order of $\omega_S = \OO_S(K_S)$.} Note that, 
  when $k \geq 2$, 
\[
\max \{4 + 1, 2k+2\} = 2k+2 = \max\{4, 2k+2\}\, .
\]
In this sense,
Theorems  \ref{thmC} and \ref{thmC'}
 generalize Gallego-Purnaprajna result to higher dimensions, and to positive characteristic.

Concerning the organization of the material: 
 in  \S \ref{S1} and \S \ref{S2} we recall some notations and useful results. 
The proof of Theorem \ref{thmM} is the content of \S \ref{S3}. It is inspired by \cite{gapu}, and it works by combining  Castelnuovo-Mumford regularity (more specifically,  Mumford's Lemma  \eqref{mumlemma}), along with more recent results of Pareschi-Popa \cite{ppI, ppIII} and Chintapalli-Iyer \cite{chiIy} on vanishing properties and global generation criteria for sheaves 
 on abelian varieties.  
The proof of Theorem \ref{thmC} 
is quite similar, although certain difficulties arise.  
On the other hand, the proof of Theorem \ref{thmY} (and of Theorem \ref{thmC'}) is considerably simpler from the technical side. 
All of them will be  given, in alphabetical order, in \S \ref{S4}.

\vskip0.3truecm\noindent\textbf{Acknowledgment.} 
 The author  whish to thank Beppe Pareschi and Sofia Tirabassi for  conversations around these topics. 

\section{Syzygies of projective varieties: the linearity property $(N_k)$}\label{S1}
In this section, the characteristic of the base field $\K$ is allowed to be arbitrary.
 Let $Y$ be a projective variety over $\K$,  and $L$ be an ample line bundle on $Y$.
The \emph{section algebra} 
$$R(L) := \bigoplus_{n \geq 0} H^0(Y, nL)$$
of $L$ is   a finitely generated  module over the polynomial ring $S_L := \mathrm{Sym} (H^0(Y, L))$, and,  hence, it admits a  \emph{minimal graded free} resolution
\[
0 \rightarrow E_d(L) \rightarrow \ldots \rightarrow E_1(L) \rightarrow E_0(L) \rightarrow R(L) \rightarrow 0,
\]
which is unique up to isomorphism.
\begin{definition}[\cite{grla}]
Given an integer $k \geq 0$, 
$L$ is said to \emph{satisfy the property} $(N_k)$ if the first $k$ steps of the minimal graded free resolution of the $S_L$-algebra $R(L)$ are linear, i.e., of the following form:
\[
E_0(L) = S_L \quad  \mathrm{and} \quad  E_i(L) = \bigoplus S_L(- (i + 1)) \quad \mathrm{for\,\, all\,\,} 1 \leq i \leq k.
\]
\end{definition}
\noindent We also refer the reader to \cite[Chapter 1.8.D]{laI}, 
or to the upcoming book of Ein and
Lazarsfeld on syzygies \cite{elbook}. 

\subsection{The geometric interpretation}\label{geointer} 
The property $(N_0)$ means that $L$ is projectively normal, that is, the multiplication maps
\begin{equation}\label{multr}
H^0(Y, L) \otimes H^0(Y, L^h) \to H^0(Y, L^h)
\end{equation}
are surjective for all $h \geq 1$.
 Note that if $L$ is  projectively normal  then $L$  is very ample (see \cite[pp.\! 38-39]{mum2}\footnote{A projectively normal line bundle is called \emph{normally generated} in \cite{mum2}, from which the notation $(N_0)$.}). In this case, $$\Ker[E_0(L) = S_L \twoheadrightarrow R(L)] = I_{Y/\mathbb{P}}$$ is  
the homogeneous ideal of $Y \hookrightarrow \mathbb P := \mathbb{P}(H^0(Y, L)^{\vee})$.

If $L$ satisfies $(N_1)$, then
$$\ldots \to \bigoplus S_L(-2)  \to I_{Y/\mathbb{P}} \to 0$$ is
 a resolution of $I_{Y/\mathbb{P}}$. So, the property $(N_1)$ for $L$ means that 
$L$ is projectively normal and the homogeneous ideal $I_{Y/\mathbb P}$ is generated by a minimal set of \emph{quadrics} $\{ q_j\}_j$. 

The property $(N_2)$ asks for the resolution to be
$$\ldots \to \bigoplus S_L(-3) \to \bigoplus S_L(-2)  \to I_{Y/\mathbb{P}} \to 0.$$ 
 This says  that the relations (or syzygies) among the quadrics $q_j$ are generated by \emph{linear} ones, that is, by those of the forms 
\[
\sum_j l_j \cdot q_j = 0, 
\]
with $l_j$ of degree $1$.
 More in general, for any $k \geq 2$, we are asking that
the first $(k-1)$ modules of syzygies among these quadrics are linear, i.e., as simple as possible.

\subsection{A  cohomological criterion}\label{cohomcriterion} 
Let $M_L$ be the \emph{kernel bundle} of an ample and \emph{globally generated} line bundle $L$ on $Y$, that is the kernel of the evaluation morphism of global sections of $L$: $$0 \to M_L \to H^0(Y, L) \otimes \OO_Y \xrightarrow{\rm{ev}} L \to 0. $$
\begin{proposition}\label{cohomcri} 
Given $k \geq 0$, 
if the vanishing
\begin{equation*}\label{cohomcrivan}
H^1(Y, M_L^{\otimes (i+1)} \otimes L^h) = 0
\end{equation*}
holds true for all integers $0 \leq i \leq k$ and
$h \geq 1$, then $L$ satisfies the property $(N_k)$.
\end{proposition}

\noindent This fact is well-known when $\Char(\K) = 0$ (see, e.g., \cite[pp.\! 510-511]{lasampl} or \cite[Proof of Theorem 4.3]{pa1}).  
It holds true as well in arbitrary characteristic thanks to
 an algebraic result of Kempf \cite{ke}, as proved by the author in \cite[\S 4]{ca} (see also \cite[footnote 2 at p.\! 1361]{ca2}).

\section{Propaedeutic results on abelian varieties}\label{S2}
As before, let $\K$ be an algebraically closed field of arbitrary characteristic. Let $A$ be an abelian variety, defined over $\K$.

\begin{definition}\label{def3.1}
A coherent sheaf $\FF$ on $A$ is said to be:
\begin{itemize}
\item[a)] $IT(0)$ (or to satisfy the \emph{Index Theorem with index} $0$), if $$H^j(A, \FF \otimes \alpha) = 0$$
for all $j > 0$ and  all closed points $\alpha \in \Pic0 A$;  \\

\item[b)]
 \emph{GV} (or a \emph{generic vanishing sheaf}), if
\begin{equation*}\label{defgv}
\mathrm{codim}_{\Pic0 A} \{\alpha \in \Pic0 A \ |\ h^j(A, \FF \otimes \alpha) \neq 0  \} \geq j
\end{equation*}
for all $j > 0$.
\end{itemize}
\end{definition}
 
\begin{remark}\label{rmk3.1}
 Of course, by definition, an $IT(0)$ sheaf is $GV$, and being $IT(0)$/$GV$ is invariant under the tensor product by a fixed element in $\Pic0 A$.  
Note that the $GV$ condition implies, in particular, that $h^j(A, \FF \otimes \alpha) = 0$ for $j > 0$ and a general\footnote{Indeed, the cohomological loci in Definition \ref{def3.1}(b) are Zariski closed, by upper semicontinuity.} $\alpha \in \Pic0 A$. Moreover, as soon as $h^j(A, \FF \otimes \alpha) = 0$ for all $j \geq 2$ and for all $\alpha \in \Pic0 A$ (and this often happens in our cases), in order to check the $GV$-ness of $\FF$ it suffices to find an $\alpha_0 \in \Pic0 A$ such that  $h^1(A, \FF \otimes \alpha_0) = 0$.   
\end{remark}

\begin{example}\label{ex3.1}
Basic (and important) examples of $IT(0)$ (resp.\! $GV$) sheaves are ample (resp.\! nef) \emph{line bundles} on $A$. Indeed, all the higher cohomology groups of an ample line bundle on an abelian variety vanish by Mumford's vanishing theorem \cite[\S 16]{mum3}. For the nefness, see \cite[Theorem 5.2]{jipa} and \cite[Example 2.1]{ito0}. 
\end{example}

An $IT(0)$  sheaf on an abelian variety is ample \cite{debarre}, and a $GV$ one is nef \cite{ppIII}.
 In fact, $IT(0)$/$GV$ sheaves satisfy certain  properties which are \emph{formally} analogous to those valid for ample/nef sheaves.   
We collect below a particularly useful one, for later use:
\begin{proposition}[\cite{ppIII}, Proposition 3.1 and Theorem 3.2]\label{presvan}
Let $\FF$ and $\GG$ be coherent sheaves on $A$, one of them locally free. If $\FF$ is $IT(0)$ and $\GG$ is  $GV$, then $\FF \otimes \GG$ is $IT(0)$. If $\FF$ and $\GG$ are both $GV$, then $\FF \otimes \GG$ is $GV$.
\end{proposition}

\section{Proof of Theorem \ref{thmM}}\label{S3}
Let $S$ be a bielliptic surface defined over an algebraically closed field $\K$, with $\Char (\K) \neq 2, 3$.
Let 
\[
\pi \colon A \to S
\] 
be the canonical cover of $S$, that is the finite \'etale cover defined by the canonical line bundle $\omega_S = \OO_S(K_S)$ of $S$, which is torsion of order $n = 2, 3, 4,$ or $6$ (see \cite[\S 3]{bomu}, or \cite[\S 10]{badescu}). 
 Note that, by our assumption, $\Char(\K)$ does not divide $n = \deg\, \pi$.
 One has that $A$ is an abelian surface (see \cite[\S 1.4, p.\! 23]{boada}), and 
\begin{equation}\label{split00}
\pi_*\OO_A \simeq \bigoplus_{r=0}^{n-1} \omega_S^{r}\, . 
\end{equation}

Take  an ample line bundle $L$  on $S$.  
Let us denote
\[
L_m :=  \OO_S(K_S + mL). 
\]
\begin{remark}\label{noK}
Note that $\pi^*L_m  \simeq \pi^*\OO_S(mL)$, as $\pi$ is \'etale. The  arguments below also apply to $\OO_S(mL)$ instead of $L_m$, with basically no modifications.
\end{remark}
\begin{remark}\label{reider}
By Reider's theorem (which holds true as well for bielliptic surfaces of positive characteristic $\neq 2, 3$ by \cite[Corollary 8]{sb}), the tensor product of two ample line bundles on $S$ is globally generated (see, e.g.,  \cite[Lemma 2.7]{gapu}). 
\end{remark}
 We assume, from now on, $m \geq 2$. So $L_m$ is globally generated, and  
 we may consider the kernel bundle $M_{L_m}$ associated to it:
\begin{equation}\label{MLM}
0 \to M_{L_m} \to H^0(S, L_m) \otimes \OO_S \to L_m \to 0\, .
\end{equation}
Thanks to Proposition \ref{cohomcri}, given $k \geq 0$,  the property $(N_k)$ for $L_m$ follows if the vanishing 
\begin{equation*}
H^1(S, M_{L_m}^{\otimes (i+1)} \otimes L_m^h) = 0
\end{equation*}
holds true for all integers $0 \leq i \leq k$ and $h \geq 1$. 
 Since, by \eqref{split00}, $M_{L_m}^{\otimes (i+1)} \otimes L_m^h$ is a direct summand of
$$\pi_*\OO_A \otimes M_{L_m}^{\otimes (i+1)} \otimes L_m^h \simeq \pi_* (\pi^*M_{L_m}^{\otimes (i+1)} \otimes \pi^*L_m^h)\, ,$$ and since the morphism $\pi$ is finite, it is more than enough  to prove that $\pi^*M_{L_m}^{\otimes (i+1)} \otimes \pi^*L_m^h$ is an $IT(0)$ sheaf on $A$, or simply that $H^1(A, \pi^*M_{L_m}^{\otimes (i+1)} \otimes \pi^*L_m^h)$. Indeed, if this is the case, one would have
\[
h^1(S, M_{L_m}^{\otimes (i+1)} \otimes L_m^h) \leq h^1(S, \pi_* (\pi^*M_{L_m}^{\otimes (i+1)} \otimes \pi^*L_m^h)) = h^1(A, \pi^*M_{L_m}^{\otimes (i+1)} \otimes \pi^*L_m^h) = 0\, .
\] 
Further, by Proposition \ref{presvan} and Example \ref{ex3.1},  we may typically focus only on the $h = 1$ case.

\subsection{Projective normality}\label{S3.1}
Take $k = 0$. We prove, more generally,  Proposition \ref{lemmap0} below, which will be also useful in the next section. 
 Its proof uses Castelnuovo-Mumford regularity, likewise \cite{gapu}.

 Let us recall that a coherent sheaf $\FF$ on a projective variety $Y$ is said to be \emph{regular} with respect to an ample and globally generated  line bundle $P$ on $Y$,
if
\[
H^j(Y, \FF \otimes P^{-j}) = 0
\] 
for all $j > 0$. Generalizing a lemma of Castelnuovo, Mumford proved that
 the multiplication map
\begin{equation}\label{mumlemma}
H^0(Y, \FF) \otimes H^0(Y, P) \to H^0(Y, \FF \otimes P)
\end{equation}
is surjective, if $\FF$ is a regular sheaf on $Y$ with respect to $P$ (\cite{mum2},  or,  e.g., \cite[Theorem 1.8.5]{laI}).

\begin{proposition}\label{lemmap0} The sheaf $\pi^*M_{L_m} \otimes \pi^*L_2$
is $GV$ as soon as  $m \geq  4$.
Hence, $\pi^*M_{L_m} \otimes \pi^*L_n$
is
 $IT(0)$ if 
 $m \geq 4$ and $n \geq 3$.
\end{proposition}
\begin{proof}
Thanks to
 Proposition \ref{presvan} and Example \ref{ex3.1}, it suffices to prove the first statement.  

To show that $\pi^*M_{L_m} \otimes \pi^*L_2$
is $GV$, we need  to find a point $\alpha_0 \in \Pic0 A$ such that $h^1(A, \pi^* M_{L_m} \otimes \pi^*L_2 \otimes \alpha_0) = 0$.
 Let us explain why: 
by taking the  short exact sequence defining $M_{L_m}$ and pulling it back via $\pi$, one gets
\[
0 \to \pi^*M_{L_m} \to H^0(S, L_m) \otimes \OO_A \to \pi^*L_m \to 0\, .
\] 
Tensoring it with $\pi^*L_2$ and taking the long exact sequence in cohomology, one has, since $\pi^*L_2$ and $\pi^*L_m$ are ample, that  $h^i(A, \pi^* M_{L_m} \otimes \pi^*L_2 \otimes \alpha) = 0$ for all $i \geq 2$ and all $\alpha \in \Pic0 A$. Therefore, as observed in Remark \ref{rmk3.1}, to get the $GV$ condition one only needs  to 
 check that $$\{ \alpha \in \Pic0 A \ |\ h^1(\pi^* M_{L_m} \otimes \pi^*L_2 \otimes \alpha) \neq 0\}$$ is \emph{properly} contained in $\Pic0 A$.

To do so, we proceed as follows. 
Since $\pi$ is finite,  
\begin{equation}\label{split1}
h^1(A, \pi^* M_{L_m} \otimes \pi^*L_2 \otimes \pi^*\beta) = h^1(S, \pi_*\pi^* (M_{L_m} \otimes L_2 \otimes \beta))
\end{equation}
 for any $\beta \in \Pic0 S$, and,  by the projection formula, 
\begin{equation}\label{split0}
\pi_* \pi^*(M_{L_m} \otimes L_2 \otimes \beta) \simeq (M_{L_m} \otimes L_2 \otimes \beta) \otimes \pi_*\OO_A \simeq \bigoplus_{r=0}^{n-1} M_{L_m} \otimes L_2 \otimes \beta \otimes \omega_S^{r}\, .
\end{equation}
Therefore, to get the vanishing $h^1(A, \pi^* M_{L_m} \otimes \pi^*L_2 \otimes \pi^*\beta_0) = 0$ for a certain $\beta_0 \in \Pic0 S$, it suffices to choose $\beta_0 \in \Pic0 S$ in  such a way  that
\begin{equation}\label{mum00}
H^1(S, L_m \otimes (L_2 \otimes \beta_0 \otimes \omega_S^{r})^{-1}) = 0 \quad \mathrm{and} \quad H^2(S, L_m \otimes (L_2 \otimes \beta_0 \otimes \omega_S^{r})^{-2}) = 0
\end{equation}
for all $r = 0, \ldots, n-1$, 
and then apply \eqref{mumlemma}.\footnote{Note that $L_2 \otimes \beta_0 \otimes \omega_S^{r}$ is globally generated for all $r$, thanks to Remark \ref{reider}.} This would give that the multiplication map
\[
H^0(S, L_m) \otimes H^0(S, L_2 \otimes \beta_0 \otimes \omega_S^{r}) \to H^0(S, L_{m} \otimes L_2 \otimes \beta_0 \otimes \omega_S^{1+ r})
\]
is surjective for all $r = 0, \ldots, n-1$, and it follows from the long exact sequence in cohomology associated to the short exact sequence defining $M_{L_m}$, twisted by $L_2 \otimes \beta_0 \otimes \omega_S^{r}$, and thanks to the vanishing $h^1(S, L_2 \otimes \beta_0 \otimes \omega_S^{r}) = 0$,\footnote{We point out that here we are \emph{not} appealing to Kodaira's vanishing. It suffices to note that 
\begin{equation}\label{rmk3.11}
H^1(S, L_2 \otimes \beta_0 \otimes \omega_S^{r}) \subseteq H^1(S, \pi_* \pi^*(L_2 \otimes \beta_0 \otimes \omega_S^{r})) = H^1(A, \pi^*(L_2 \otimes \beta_0 \otimes \omega_S^{r})) 
\end{equation}
by \eqref{split00}, and to apply Mumford's vanishing on abelian varieties (see Example \ref{ex3.1}).
} that $h^1(S, M_{L_m} \otimes  L_2 \otimes \beta_0 \otimes \omega_S^{r}) = 0$ for all $r = 0, \ldots, n-1$. Hence, by \eqref{split0} and \eqref{split1}, we finally get $h^1(A, \pi^* M_{L_m} \otimes \pi^*L_2 \otimes \pi^*\beta_0) = 0$. 

Now, the left-hand vanishing in \eqref{mum00} holds true for any $\beta_0$ and $r$, because
 $L_m \otimes (L_2 \otimes \beta_0 \otimes \omega_S^{r})^{-1}$ is ample on $S$ (and we can argue as in \eqref{rmk3.11}). The right-hand vanishing in \eqref{mum00}  holds true as well if $m \geq 5$,  by the same reason. 
So we reduced to find a $\beta_0 \in \Pic0 S$ such that that $H^2(S, L_4 \otimes (L_2 \otimes \beta_0 \otimes \omega_S^{r})^{-2})$ vanishes for all $r = 0, \ldots, n-1$. By Serre duality,
\begin{equation}\label{serre0}
h^2(S, L_4 \otimes (L_2 \otimes \beta_0 \otimes \omega_S^{r})^{-2}) = h^0(S, \omega_S^{\, (2 + 2r)}   \otimes \beta_0^2 )\, .
\end{equation}
Let us recall now that, as a consequence of the canonical bundle formula, the canonical line bundle of a bielliptic surface is not only numerically trivial, but $\omega_S \in \Pic0 S$ (see, e.g., \cite[Corollary 1.15, at p.\! 18]{boada}).
Therefore, 
 if we  chose $\beta_0 \in \Pic0 S$ such that 
\[
\beta_0^{2} \neq  \omega_S^{\, -(2+2r)}
\] for all $r = 0, \ldots, n-1$ (this is possible as  $\Pic0 S$ is a non-trivial abelian variety under our assumption on $\Char(\K)$, see \cite[Remark 1.2, p.\! 14]{boada}), we get that the   right-hand side of \eqref{serre0} has to be $0$, because $\omega_S^{\, (2 + 2r)}    \otimes \beta_0^2 \in \Pic0 S \setminus \{\OO_S\}$.
\end{proof}

\begin{corollary}
The multiplication maps
\[
H^0(S, K_S + mL) \otimes H^0(S, K_S +nL) \to H^0(S, 2K_S + (m+n)L) 
\]
are surjective, for all $n, m \geq 3$ and $n+m \geq 7$. 
\end{corollary}

\subsection{The case $\bf{k = 1}$}\label{S3.21}
Our aim is now to show that  
\begin{equation}\label{gc11}
H^1(S, M_{L_m}^{\otimes 2} \otimes L_m^h) = 0
\end{equation} for all $h \geq 1$, 
  as soon as $m \geq 4$.
Note that
\[
\pi^*M_{L_m}^{\otimes 2} \otimes \pi^*L_m^h = \pi^*(M_{L_m} \otimes L_2)  \otimes \pi^*(M_{L_m} \otimes L_2) 
\otimes \pi^*(L_m \otimes L_2^{-2}) \otimes \pi^*L_m^{(h-1)}\, . 
\]
Since $\pi^*(M_{L_m} \otimes L_2)$ is $GV$ by Proposition \ref{lemmap0}, 
we already know, by applying Proposition   \ref{presvan}, that $\pi^*M_{L_m}^{\otimes 2} \otimes \pi^*L_m^h$ is $IT(0)$ if $m > 4$, or if $m = 4$ and $h \geq 2$. Hence, as already explained, \eqref{gc11} holds in these cases. So we only need to show that $H^1(S, M_{L_4}^{\otimes 2} \otimes L_4) = 0$. We have, more generally, the following vanishings.
\begin{lemma}\label{lemmap11}
Let $m \geq 4$. Then,
 \begin{equation}\label{eqlemmap11}
H^j(S, M_{L_m}^{\otimes 2} \otimes L_4 \otimes \beta) = 0 
\end{equation}
for all $\beta \in \Pic0 S$ and $j \geq 1$.
\end{lemma}
\begin{proof}
Fix $\beta \in \Pic0 S$, and write
\[
\pi^*M_{L_m}^{\otimes 2} \otimes \pi^*L_4 \otimes \pi^*\beta = \pi^*M_{L_m} \otimes \pi^*E_m\, ,
\]
where $E_m := M_{L_m} \otimes L_4 \otimes \beta$. Since $\pi^*E_m = (\pi^*M_{L_m} \otimes \pi^*L_3) \otimes \pi^*(L \otimes \beta)$ is the tensor product of an $IT(0)$ sheaf (by Proposition \ref{lemmap0}) and an ample line bundle on $A$, both coming from $S$ as pullbacks,  one has that 
$E_m$ is globally generated thanks to \cite[Corollary 4.7]{chiIy}.\footnote{Indeed, as already observed, if $n$ is the order of $\omega_S$, then $\Z_n$ acts freely on $A$, and 
\[
\pi \colon A \to A/\Z_n \simeq S
\] 
is the quotient morphism (see, e.g.,  \cite[\S 1.4, p.\! 23]{boada}). Then, \cite[Corollary 4.7]{chiIy} (which is an equivariant version of  the main result of \cite{ppI}) says that  $\pi^*E_m$ is $\Z_n$-globally generated. This means, by definition, that the evaluation morphism
\[
H^0(A, \pi^*E_m)^{\Z_n} \otimes \OO_A \to \pi^*E_m
\] 
is surjective. Since $H^0(A, \pi^*E_m)^{\Z_n} \simeq H^0(S, E_m)$, one  gets that $H^0(S, E_m) \otimes \OO_S \to E_m$ is surjective, too.
Finally, note that the authors of \cite{chiIy} assume characteristic zero, however their proof is algebraic, and it works  well if the characteristic of the base field is coprime with $|\Z_n|$, as in our case.} 
Let us consider now the kernel bundle associated to $E_m$
\[
0 \to M_{E_m} \to H^0(S, E_m) \otimes \OO_S \to E_m \to 0\, ,
\]
and 
the multiplication map
\[
H^0(S, E_m) \otimes H^0(S, L_m) \xrightarrow{f_m} H^0(S, E_m \otimes L_m)\, .
\]
If 
\begin{equation}\label{eqlemmap21}
H^1(S, M_{E_m} \otimes L_m) = 0\, , 
\end{equation}
then $f_m$ is surjective, and using the short exact sequence defining $M_{L_m}$ twisted by $E_m$, that is, 
\begin{equation}\label{eqlemmap211}
0 \to M_{L_m} \otimes E_m \to H^0(S, L_m) \otimes E_m \to L_m \otimes E_m \to 0\, ,
\end{equation}
this implies that 
 $H^1(S, M_{L_m} \otimes E_m) = 0$, as we know that $H^1(S, E_m) = 0$.\footnote{
Indeed,   
 $\pi^*E_m$ is $IT(0)$, and this gives the desired vanishing.
}
 So,  \eqref{eqlemmap21} would imply
\eqref{eqlemmap11} with $j = 1$. On the other hand, when $j \geq 2$,  from \eqref{eqlemmap211} we get  
\[
H^{j-1}(S, L_m \otimes E_m) \to H^j(S, M_{L_m} \otimes E_m) \to H^0(S, L_m) \otimes H^j(S, E_m)\, , 
\]
and both the left-hand side and the right-hand side are $0$ by Proposition \ref{presvan}.

Now, in order to prove \eqref{eqlemmap21}, we claim that $\pi^*M_{E_m} \otimes \pi^*L_2$ is $GV$. In this way, thanks to Proposition \ref{presvan}, $\pi^*M_{E_m} \otimes \pi^*L_m$ is $IT(0)$, and hence, in particular, we obtain \eqref{eqlemmap21}. So let us  
 prove the claim: like before, we will show the existence of a certain $\beta_0 \in \Pic0 S$ such that
\begin{equation}\label{eqlemmap31}
H^1(S, M_{E_m} \otimes L_2 \otimes \beta_0 \otimes \omega_S^r) = 0
\end{equation}
for all $r = 0, \ldots, n$, as  this suffices to get the $GV$-ness of $\pi^*M_{E_m} \otimes \pi^*L_2$. 
Fix $r = 0, \ldots, n$. 
The vanishing in \eqref{eqlemmap31} follows from the surjectivity of the following multiplication map
\[
 H^0(S, E_m) \otimes H^0(S, L_2 \otimes \beta_0 \otimes \omega_S^r) \to H^0(S, E_m \otimes L_2 \otimes \beta_0 \otimes \omega_S^r)\, ,
\]
which, thanks to the Mumford's Lemma \eqref{mumlemma}, is in turn a consequence of the vanishings:
\begin{enumerate}
\item[\emph{1)}] $H^1(S, E_m \otimes (L_2 \otimes \beta_0 \otimes \omega_S^r)^{-1}) = 0$, and \\
\item[\emph{2)}] $H^2(S, E_m \otimes (L_2 \otimes \beta_0 \otimes \omega_S^r)^{-2}) = 0$.
\end{enumerate}
We just need to note that \emph{(1)} holds  as, by definition,  $E_m \otimes (L_2 \otimes \beta_0 \otimes \omega_S^r)^{-1} = M_{L_m} \otimes L_2 \otimes \left(\beta \otimes  \beta_0^{-1} \otimes \omega_S^{-r}\right)$.  Since  in the proof of Proposition \ref{lemmap0} we   showed that $H^1(S, M_{L_m} \otimes L_2 \otimes \gamma)=0$ for a general $\gamma \in \Pic0 S$, we have done as $\beta \otimes \beta_0^{-1} \otimes \omega_S^{-r} \in \Pic0 S$, and we may take $\beta_0 \in \Pic0 S$ general. 
On the other hand, \emph{(2)} holds true as well, because $E_m \otimes (L_2 \otimes \beta_0 \otimes \omega_S^r)^{-2} = M_{L_m} \otimes \left(\beta \otimes \beta_0^{-2} \otimes \omega_S^{-(1+2r)}\right)$, and $H^2(S, M_{L_m} \otimes \gamma) = 0$ for a general  $\gamma \in \Pic0 S$.    
\end{proof}

\subsection{The general case}\label{S3.2}
To conclude the proof of Theorem \ref{thmM}, we  prove the following result:
\begin{proposition}
Let $i \geq 1$ be an integer.  One has that
\begin{equation}\label{gc1}
H^1(S, M_{L_m}^{\otimes (i+1)} \otimes L_m^h) = 0\, ,
\end{equation}
  if $m \geq 2i +2$ and $h \geq 1$.
\end{proposition}

Note that, like before, since
\[
\pi^*M_{L_m}^{\otimes (i+1)} \otimes \pi^*L_m^h =\underbrace{\pi^*(M_{L_m} \otimes L_2) \otimes \ldots \otimes \pi^*(M_{L_m} \otimes L_2)}_{i+1} 
\otimes\, \pi^*(L_m^h \otimes L_2^{-(i+1)})\, , 
\]
and 
since $\pi^*(M_{L_m} \otimes L_2)$ is $GV$ by Proposition \ref{lemmap0}, 
 $\pi^*M_{L_m}^{\otimes (i+1)} \otimes \pi^*L_m^h$  
is  $IT(0)$ when $m > 2i+2$, or if $m = 2i+1$ and $h >1$, thanks to Proposition  \ref{presvan}. On the other hand, if $m = 2i+2$ and $h = 1$, it is just
 $GV$, and such fact a priori does not  suffice to get  \eqref{gc1}.
 Therefore, we want to prove, with a more direct argument, that
\begin{proposition}\label{lemmap1}
 \begin{equation}\label{eqlemmap1}
H^j(S, M_{L_m}^{\otimes (i+1)} \otimes L_{2i+2} \otimes \beta) = 0
\end{equation}
for all $\beta \in \Pic0 S$ and $j \geq 1$, if $m \geq 4$.
\end{proposition}
\begin{proof}
We argue by induction on $i$. The base step  ($i = 1$) is precisely Lemma \ref{lemmap11} above, and actually, thanks to the inductive hypothesis, the proof for $i \geq 2$ is basically the same. 
Indeed, let us write
\[
\pi^*M_{L_m}^{\otimes (i+1)} \otimes \pi^*L_{2i+2} \otimes \pi^*\beta = \pi^*M_{L_m} \otimes \pi^*E_m\, ,
\]
where $E_m := M_{L_m}^{\otimes i} \otimes L_{2i+2} \otimes \beta$.  As before, let us observe that $E_m$ is globally generated: since 
\[
\pi^*E_m = 
\left[\left(\pi^*M_{L_m} \otimes \pi^*L_3\right) \otimes \left(\pi^*M_{L_m} \otimes \pi^*L_2\right)^{\otimes (i-1)}\right] \otimes (\pi^*L_{2i+2} \otimes \pi^*L_{3+2(i-1)}^{-1} \otimes \pi^*\beta)\, ,
\] 
$\pi^*E_m$ can be written as a tensor product of an $IT(0)$ sheaf on $A$ and an ample line bundle on $A$. Indeed,
 $\pi^*M_{L_m} \otimes \pi^*L_3$ is $IT(0)$ by Proposition \ref{lemmap0} and $\left(\pi^*M_{L_m} \otimes \pi^*L_2\right)^{\otimes (i-1)}$ is $GV$ by Propositions \ref{lemmap0} and \ref{presvan}, hence their product is $IT(0)$. Moreover, $\pi^*L_{2i+2} \otimes \pi^*L_{3+2(i-1)}^{-1} \otimes \pi^*\beta = \pi^*(L \otimes \beta)$, which is ample. Therefore, as explained above (see the footnote 8 at p.\! 10),  
 $E_m$ is globally generated by \cite[Corollary 4.7]{chiIy}. 
Let us consider the multiplication map
\[
H^0(S, E_m) \otimes H^0(S, L_m) \xrightarrow{f_m} H^0(S, E_m \otimes L_m).
\]
If 
\begin{equation}\label{eqlemmap2}
H^1(S, M_{E_m} \otimes L_m) = 0\, , 
\end{equation}
then $f_m$ is surjective, and, hence, \eqref{eqlemmap1} holds true for $j = 1$, as $H^1(S, E_m) = 0$. 
The argument for $j \geq 2$ is the same as before, and we do not reproduce it here. 

In order to get \eqref{eqlemmap2}, it suffices to note  that $\pi^*M_{E_m} \otimes \pi^*L_2$ is $GV$ (hence, $\pi^*M_{E_m} \otimes \pi^*L_m$ is $IT(0)$).  
This is again an application of Mumford's Lemma:  we  simply show that
\begin{equation}\label{eqlemmap3}
H^1(S, M_{E_m} \otimes L_2   \otimes   \omega_S^r) = 0
\end{equation}
for all $r = 0, \ldots, n$, as this is enough to say that $\pi^*M_{E_m} \otimes \pi^*L_2$ is $GV$. 
Fix $r = 0, \ldots, n$. 
The vanishing in \eqref{eqlemmap3} follows from the surjectivity of the  multiplication map
\[
H^0(S, E_m) \otimes H^0(S, L_2  \otimes   \omega_S^r) \to H^0(S, E_m \otimes L_2   \otimes   \omega_S^r)\, ,
\]
which in turn follows from the vanishings:
\begin{enumerate}
\item[\emph{1)}] $H^1(S, E_m \otimes (L_2   \otimes   \omega_S^r)^{-1}) = 0$, and \\
\item[\emph{2)}] $H^2(S, E_m \otimes (L_2   \otimes  \omega_S^r)^{-2}) = 0$.
\end{enumerate}
Here we use  the inductive hypothesis. Indeed, 
 by definition,  
$E_m \otimes (L_2   \otimes  \omega_S^r)^{-1} = M_{L_m}^{\otimes i} \otimes L_{2i} \otimes \beta    \otimes \omega_S^{-r}$, and  since   $\beta    \otimes \omega_S^{-r} \in \Pic0 S$, \emph{(1)} holds true.
Moreover, \emph{(2)} holds true as well, because $E_m \otimes (L_2   \otimes \omega_S^r)^{-2} = M_{L_m}^{\otimes i} \otimes L_{2i-2} \otimes \beta  \otimes \omega_S^{-(1+2r)}$, and $H^2(S, M_{L_m}^{\otimes i} \otimes L_{2i-2} \otimes \gamma) = 0$ for  all $\gamma \in \Pic0 S$.    
This follows from  the long exact sequence in cohomology given by the short exact sequence defining $M_{L_m}$, twisted by $M_{L_m}^{\otimes (i-1)} \otimes L_{2i-2} \otimes \gamma$. More explicitly, from
\[
0 \to M_{L_m}^{\otimes i} \otimes L_{2i-2} \otimes \gamma \to H^0(S, L_m) \otimes 
M_{L_m}^{\otimes (i-1)} \otimes L_{2i-2} \otimes \gamma \to M_{L_m}^{\otimes (i-1)} \otimes L_{2i-2} \otimes L_m \otimes \gamma \to 0\, ,
\]
we get, since  $H^j(S, M_{L_m}^{\otimes (i-1)} \otimes L_{2i-2} \otimes L_m \otimes \gamma) = 0$ for all $j \geq 1$, that 
\[
H^2(S, M_{L_m}^{\otimes i} \otimes L_{2i-2} \otimes \gamma) \simeq 
H^0(S, L_m) \otimes 
H^2(S, M_{L_m}^{\otimes (i-1)} \otimes L_{2i-2} \otimes \gamma)\, .
\]
By repeating the same reasoning $i-1$ times, we finally get that the left-hand side (which is the cohomology group whose vanishing we are interested in) is isomorphic to
\[
H^0(S, L_m)^{\otimes i} \otimes 
H^2(S, L_{2i-2} \otimes \gamma) = 0\, .
\]
\end{proof}

\section{Results in higher dimension}\label{S4}

\subsection{Proof of Theorem \ref{thmY}}
We now  work in characteristic zero. Let 
$Y$ be a smooth projective varieties of dimension $d$ with $K_Y \equiv 0$, $P$ be an ample and globally generated line bundle on $Y$, and $N$ be a numerically trivial line bundle on $Y$ such that $N +P$ is globally generated. Given a positive integer $m$,  let us define 
\[
P_m := N + mP\, .
\]  

 Theorem \ref{thmY} follows from the next result, thanks to Proposition \ref{cohomcri}. 
 Let $i \geq 0$, $h \geq 1$, and $m, m' \geq d+1$.
\begin{proposition}\label{propY}
One has 
\[
H^j(Y, M_{P_m}^{\otimes (i+1)} \otimes P_{m'}^h) = 0
\]
for all $j \geq 1$, if $m' \geq i+1$.
\end{proposition}
\begin{proof}
We may suppose, for simplicity,  that $h = 1$. Otherwise the proof is basically the same.
We prove it by induction on $i$. If $i = 0$, we want  $H^j(Y, M_{P_m} \otimes P_{m'}) = 0$, when $m' \geq 1$. Since, by Kodaira's vanishing, $H^j(Y, P_{m'}) = H^j(Y, K_Y \otimes  (P_{m'}  \otimes  K_Y^{-1})) = 0$ and $H^j(Y, P_m \otimes P_{m'}) = 0$, we only need to show that the multiplication map
\[
H^0(Y, P_m) \otimes H^0(Y, P_{m'}) \to H^0(Y, P_m \otimes P_{m'}) 
\]
surjects. This map sits in the following commutative diagram
\[
\xymatrix{
H^0(P_{m})  \otimes \left(H^0(N  \otimes P) \otimes \overbrace{H^0(P) \otimes \ldots \otimes H^0(P)}^{m'-1}\right) \ar[d] \ar[dr] & \\
H^0(P_m) \otimes H^0(P_{m'}) \ar[r] &H^0(P_m \otimes P_{m'})\, ,
}
\]
and hence its surjectivity follows by iteratively apply  the next Lemma to the diagonal map.
\begin{lemma}

\noindent \emph{i)} The multiplication map
\[
H^0(Y, P_m) \otimes H^0(Y, N  \otimes P) \to H^0(Y, P_{m} \otimes N  \otimes P)
\]
is surjective. 

\noindent \emph{ii)} For all $k_1, k_2 \geq 0$, the multiplication map
\[
H^0(Y, P_m  \otimes (N  \otimes P)^{k_1} \otimes P^{k_2}) \otimes H^0(Y, P) 
 \to H^0(Y, P_{m} \otimes (N  \otimes P)^{k_1} \otimes P^{k_2 + 1})
\]
is surjective.
\end{lemma}
\begin{proof}
Let us prove only (ii), assuming $k_1 = k_2= 0$ for simplicity. The desired surjectivity follows from Mumford's lemma \eqref{mumlemma}, if $H^j(Y, P_{m}  \otimes  \OO_Y(P)^{-j}) = 0$ for all $j = 1, \ldots, d$. But $P_{m}  \otimes  \OO_Y(P)^{-j} = \OO_Y(N + (m-j)P)$, and $m - j \geq m - d \geq 1$. Hence, the vanishings follows from Kodaira's.
\end{proof}

We now assume $i \geq 1$. Define
\[
E_{m, m'} := M_{P_m}^{\otimes i} \otimes P_{m'}. 
\]
Let us consider the multiplication map
\begin{equation}\label{surjY0}
H^0(Y, E_{m, m'}) \otimes H^0(Y, P_m) \to H^0(Y, E_{m, m'} \otimes P_m)\, . 
\end{equation} 
Its surjectivity would imply that $0 = H^1(Y, M_{P_m} \otimes E_{m, m'}) = H^1(Y, M_{P_m}^{\otimes (i+1)} \otimes P_{m'})$, as
 $H^1(Y, E_{m, m'}) =0$  by the inductive hypothesis. Moreover, since  one actually has $H^j(Y, E_{m, m'}) = H^j(Y, E_{m, m'} \otimes P_m) =0$ for 
all $j \geq 1$, we also get $0 = H^j(Y, M_{P_m} \otimes E_{m, m'}) = H^j(Y, M_{P_m}^{\otimes (i+1)} \otimes P_{m'})$ for all $j \geq 2$. 

So, we only need to prove the surjectivity of \eqref{surjY0}. Since it sits in the  commutative diagram
\[
\xymatrix{
H^0(E_{m, m'}) \otimes \left(H^0(N+P) \otimes \overbrace{H^0(P) \otimes \ldots \otimes H^0(P)}^{m-1}\right)  \ar[d] \ar[dr] & \\
H^0(E_{m, m'}) \otimes H^0(P_m) \ar[r] &H^0(E_{m, m'} \otimes P_m)\, ,
}
\]
we may iteratively apply the next 
\begin{lemma}
The morphisms 
\[
H^0(Y, E_{m, m'}) \otimes H^0(Y, \OO_Y(N+P)) \to H^0(Y, E_{m, m'} \otimes \OO_Y(N+P))
\]
and
\begin{multline*}
H^0(Y, E_{m, m'} \otimes \OO_Y(N+P)^{k_1} \otimes \OO_Y(P)^{k_2}) \otimes H^0(Y, \OO_Y(P)) \\ 
\to H^0(Y, E_{m, m'} \otimes \OO_Y(N+P)^{k_1} \otimes \OO_Y(P)^{k_2 + 1})
\end{multline*}
are surjective, for all $k_1, k_2 \geq 0$.
\end{lemma}
\begin{proof}
As before, we only prove the surjectivity of the second map, assuming $k_1=k_2=0$ for simplicity. So, we need to show that $H^l(Y, E_{m, m'} \otimes P^{-l}) = 0$, for all $l = 1, \ldots, d$.
  For $l = 1$, we  directly apply  the inductive hypothesis, as 
	$E_{m, m'} \otimes P^{-1} = 
M_{P_m}^{\otimes i} \otimes P_{m' - 1}$ and $m'-1 \geq i$.
	If $2 \leq l \leq d$, we may argue as follows. By definition $E_{m, m'} \otimes P^{-l} = 
M_{P_m}^{\otimes i} \otimes P_{m' - l}$. From the short exact sequence
\[
0 \to M_{P_m}^{\otimes i} \otimes P_{m' - l}  \to H^0(P_m) \otimes M_{P_m}^{\otimes (i-1)} \otimes P_{m' - l}  
\to M_{P_m}^{\otimes (i-1)} \otimes P_{m' - l}  \otimes P_m \to 0\, ,
\]
since 
\begin{equation}\label{j-1}
H^{l-1}(M_{P_m}^{\otimes (i-1)} \otimes P_{m' - l}  \otimes P_m) = H^l(M_{P_m}^{\otimes (i-1)} \otimes P_{m' - l}  \otimes P_m) = 0\, ,
\end{equation} 
it follows that
\[
H^l(M_{P_m}^{\otimes i} \otimes P_{m' - l}) \simeq H^0(P_m) \otimes H^l(M_{P_m}^{\otimes (i-1)} \otimes P_{m' - l})\, .
\]
Let us justify  \eqref{j-1}: it follows again from the inductive hypothesis, as by definition
 $M_{P_m}^{\otimes (i-1)} \otimes P_{m' - l}  \otimes P_m = M_{P_m}^{\otimes (i-1)} \otimes P_{m+ m' - l}  \otimes N$, and $m+m'-l \geq m' \geq i - 1$. 

Iterating this argument, we  get
\[
H^l(M_{P_m}^{\otimes i} \otimes P_{m' - l}) \simeq (H^0(P_m))^{\otimes i} \otimes H^l(P_{m' - l})\, ,
\]
and the last cohomology is $0$ by Kodaira's vanishing, 
as $m' > d$.
\end{proof}

This concludes the proof of Proposition \ref{propY}, and hence of Theorem \ref{thmY}.
\end{proof}

\begin{remark}\label{NEF}
Note that the same proof  works  well even if we only assume that $K_Y^{-1}$ and $N^{-1}$ are both nef, and that $N + P$ is ample and globally generated. However, this result is sometimes weaker than some of those appearing in the literature.   
\end{remark}

\subsection{Proof of Theorem \ref{thmC'}}
As already observed, the tensor product of two ample line bundles on a hyperelliptic variety as in the statement of Theorem \ref{thmC'}, is globally generated by \cite[Corollary 4.7]{chiIy}. Therefore, $N + 2L = (N + L) + L$ is globally generated for any numerically trivial line bundle $N$ on $X$. 

By repeating the argument of Theorem \ref{thmY}, with $2L$ as $P$, we see that $$m \geq \max\{2d+1, 2(k+1)\}$$ suffices to obtain the property $(N_k)$ for $N + mL$. 
 The characteristic zero hypothesis in Theorem \ref{thmY} is only used when we apply Kodaira's vanishing. But, as noted above, over an hyperelliptic variety such that $\Char(\K)$ is coprime with the cardinality of the acting group $G$, the same vanishing holds true as well. Indeed, 
\[
H^j(X, N + L) \subseteq H^j(X, \pi_*\pi^*(N+L)) = H^j(A, \pi^*(N+L)) = 0 
\] 
by \cite[Corollary p.\! 72]{mum3}, and 
by Mumford's vanishing \cite[\S 16]{mum3} if $j > 0$.

\subsection{Proof of Theorem \ref{thmC}}
The proofs of this subsection are quite similar to the previous ones in \S \ref{S3}: the only differences regarding the dimensions of our varieties are related to  the Castelnuovo-Mumford regularity. This requires a slightly finer analysis.

Let $G$ be a finite  group of cardinality $|G| = n$ acting freely, and not only by translations, on an abelian variety $A$ of dimension $d$, and let
\[
\pi \colon A \to X := A/G
\]
be the quotient morphism. We still assume that $A$ (and hence $X$) is defined over an algebraically
closed field $\K$, with $\Char(\K)$ not dividing $|G|$. If $G$ is commutative, 
one has
\[
\pi_* \OO_A = \OO_X \oplus \bigoplus_{r= 1}^{n-1} \beta_r\, ,
\]
where $\beta_r$  
 are \emph{numerically trivial} line bundles on $X$  for all $r$ (see \cite[Remark at p.\! 72, and Proposition 3 at p.\! 71]{mum3}). 
We  also assume that the Picard variety $(\Pic0 X)_{red}$
of $X$ has positive dimension.

In this situation, one can argue similarly to the surface case (\S \ref{S3}), but with more work involved.
Namely, let
\[
L_m := \OO_X(N + mL)
\]
where $N$ and $L$ are  line bundles on $X$, which are numerically trivial and ample,  respectively.
By \cite[Corollary 4.7]{chiIy}, $L_m$ is globally generated as soon $m \geq 2$. Moreover, 
note that, since $\pi^*N$ is numerically trivial on $A$ and since $A$ is an abelian variety, $\pi^*N \in \Pic0 A$, and hence it can be essentially ignored most of the time.
 The analog of Proposition  
\ref{lemmap0} holds true for $L_m$, that is,
\begin{equation}\label{lemmapd}
\pi^*(M_{L_m} \otimes L_2)\quad \textrm{is $GV$\! on $A$, if $m \geq 2d$}.
\end{equation}
Indeed, 
as the reader may notice, the proof works well in this new context, except that, instead of \eqref{mum00}, now we need to find a certain $\beta_0$ in the positive dimensional abelian variety $(\Pic0 X)_{red}$ such that, for all $r$, 
\begin{multline*}
H^1(X, L_m \otimes (L_2 \otimes \beta_0 \otimes \beta_r)^{-1}) = H^2(X, L_m \otimes (L_2 \otimes \beta_0 \otimes \beta_r)^{-2}) = \\
= \ldots = H^d(X, L_m \otimes (L_2 \otimes \beta_0 \otimes \beta_r)^{-d}) = 0\, .
\end{multline*}
This is guaranteed by a similar reasoning, as we are now assuming $m \geq 2d$. 
Let us just comment on the last vanishing when $m = 2d$. Note that $L_m \otimes (L_2 \otimes \beta_0 \otimes \beta_r)^{-d} = \beta_0^{-d} \otimes \widetilde{N}$ for a certain $\widetilde{N} \equiv 0$ on $X$, independent from $\beta_0$. One has  
\begin{equation}\label{(!!)}
H^d(X, \beta_0^{-d} \otimes \widetilde{N}) = 0
\end{equation} 
for a general $\beta_0 \in (\Pic0 X)_{red}$.  
Indeed, since  $\omega_X$ and $\widetilde{N}$ are both numerically trivial line bundle,  there exists, by \cite[Exp. XIII, Th\'eor\`eme 4.6]{kleiman}, an integer $s \geq 1$ such that $\omega_X^s$ and $\widetilde{N}^{s} \in (\Pic0 X)_{red}$. 
If   
$0 \neq h^d(X, \beta_0^{-d} \otimes \widetilde{N})
 = h^0(X, \omega_X \otimes (\beta_0^{-d} \otimes \widetilde{N})^{-1})$,  
we would have that
\[
h^0(X, (\omega_X \otimes (\beta_0^{-d} \otimes \widetilde{N})^{-1})^{\otimes s}) \neq 0\, ,
\]
as the tensor product of  non-zero sections is non-zero. 
Since  $(\omega_X \otimes (\beta_0^{-d} \otimes \widetilde{N})^{-1})^{\otimes s} = \omega_X^s \otimes \widetilde{N}^{-s} \otimes \beta_0^{ds} \in (\Pic0 X)_{red}$, and since $(\Pic0 X)_{red}$ is a positive dimensional abelian variety, we get a contradiction if $\beta_0 \in (\Pic0 X)_{red}$ is general.

In particular, if $m \geq 2d$, we get 
\[
H^1(X, M_{L_m} \otimes L_m^h) = 0
\]
for all $h \geq 1$. 

 Now, let 
 $m \geq 2d$. We have
\begin{proposition}\label{mm'}
\[
H^j(X, M_{L_m}^{\otimes 2} \otimes L_{m'} \otimes \beta) = 0
\]
for all  $j \geq 1$ and  for any numerically trivial line bundle $\beta$ on $X$, if $m' \geq 2d$.
\end{proposition}
\begin{proof}
 Given $\beta \equiv 0$ on $X$,  define 
\[
E_{m, m'} := M_{L_m} \otimes L_{m'} \otimes \beta\, .
\]
Thanks to \eqref{lemmapd}, the proof is like that of Lemma \ref{lemmap11}.  As before, the only difference is that now we have to find a $\beta_0 \in (\Pic0 X)_{red}$ such that
\begin{multline*}
H^1(X, E_{m, m'} \otimes (L_2 \otimes \beta_0 \otimes \beta_r)^{-1}) = H^2(X, E_{m, m'} \otimes (L_2 \otimes \beta_0 \otimes \beta_r)^{-2}) = \\
= \ldots = H^d(X, E_{m, m'} \otimes (L_2 \otimes \beta_0 \otimes \beta_r)^{-d}) = 0\, .
\end{multline*}
We only show the last vanishing, as the other ones follows now easily,  because $m' - 2j \geq 2$ if $j < d$.   
 By definition, 
\[
E_{m, m'} \otimes (L_2 \otimes \beta_0 \otimes \beta_r)^{-d} = M_{L_m} \otimes \OO_X((m' - 2d)L)  \otimes \beta_0^{-d} \otimes \widetilde{N}\, ,
\]
for a certain $\widetilde{N} \equiv 0$ on $X$, and  the $H^d$ of the right-hand side is $0$ for a general $\beta_0 \in \Pic0 X$. 
This  follows from the short exact sequence defining the kernel bundle $M_{L_m}$, as $H^d(X, \OO_X((m' -2d)L) \otimes \beta_0^{-d} \otimes \widetilde{N}) = 0$. 
If $m' > 2d$ this is clear. When $m' = 2d$, 
to see the vanishing of 
$H^d(X, \beta_0^{-d} \otimes \widetilde{N})$ we may argue as in \eqref{(!!)}.
\end{proof}

Therefore, by Proposition \ref{cohomcri},  $L_m$ satisfies the property $(N_1)$, if $m \geq 2d$. 

\begin{remark}
Note that, when $d \geq 3$, we might more directly say that 
\begin{equation*}
\pi^*(M_{L_m}^{\otimes 2} \otimes L_{m'})\quad \textrm{is $IT(0)$\, if $m \geq 2d$  and $m' \geq 5$}\, ,
\end{equation*}
thanks to \eqref{lemmapd} and Proposition \ref{presvan}. However, 
 unlike the surface case,  the inequality $2d \leq (2k + 2) -2$ is no longer guaranteed when $k \geq 2$. This is the only obstruction that prevents a proof by induction on $k$, as the one in \S \ref{S3}, when $d \geq 3$.  
\end{remark}

On the other hand, if $2d > 2k +2$, one has that
\[
\pi^*(M_{L_m}^{\otimes (k+1)} \otimes L_{2d}) = \underbrace{\pi^*(M_{L_m} \otimes L_{2}) \otimes \ldots \otimes \pi^*(M_{L_m} \otimes L_{2})}_{k+1} \otimes\, \pi^*L_{2d} \otimes \pi^*L_{2}^{-(k+1)}
\]
is $IT(0)$ as soon as $m \geq 2d$, by  \eqref{lemmapd} and Proposition \ref{presvan}. This concludes the proof of Theorem \ref{thmC}.

\providecommand{\bysame}{\leavevmode\hbox
to3em{\hrulefill}\thinspace}


\begin{thebibliography}{EMS}


\bibitem[AKL]{akl} D. Agostini, A.  K\"uronya and V. Lozovanu, 
{Higher syzygies of surfaces with numerically trivial canonical bundle}, 
Math. Z. \textbf{293} (2019), No. 3--4, 1071--1084.



\bibitem[ALA]{ala} R. Auffarth and G. Lucchini Arteche, 
{Smooth quotients of complex tori by finite groups}, Math. Z. \textbf{300} (2022), no. 2, 1071--1091.



\bibitem[Ba]{badescu}  L. B\u{a}descu, {\em Algebraic surfaces},
 Springer-Verlag, New York, 2001.


\bibitem[BL]{bala} P. Bangere  and J. Lacini, 
{Syzygies of adjoint linear series on projective varieties},
Duke Math. J. \textbf{174} (2025), no. 3, 473--499.


\bibitem[BCF]{bacafr} I. Bauer, F. Catanese and D. Frapporti, 
{Generalized Burniat type surfaces and Bagnera-de Franchis varieties},
 J. Math. Sci., Tokyo \textbf{22} (2015), no. 1, 55--111.

\bibitem[BDN]{bedenu} P. Belmans, A. Demleitner and P. N\'u\~nez, {The Albanese morphism for hyperelliptic varieties}, 
preprint arXiv:2411.14814 (2024).





\bibitem[Bo]{boada} D. Boada De Narvaez,  {\em Moduli of bielliptic surfaces},
Ph.D. Thesis  Technische Universit\"at M\"unchen, 2021. 

\bibitem[BM]{bomu} E. Bombieri and D. Mumford, 
{Enriques' classification of surfaces in Char p, II} in {\em Complex analysis and algebraic geometry}, pp. 23--42
Iwanami Shoten Publishers, Tokyo, 1977.

\bibitem[C1]{casurv2015} F. Catanese, {Topological methods in moduli theory}, 
Bull. Math. Sci. \textbf{5} (2015), no. 3, 287--449.

\bibitem[C2]{casurvey} F. Catanese, 
{Cyclic symmetry on complex tori and Bagnera-De Franchis manifolds},
in {\em Galois covers, Grothendieck-Teichm\"uller theory and dessins d'enfants. Interactions between geometry, topology, number theory and algebra},  Proceedings from the London Mathematical Society Midlands regional meeting and workshop, Leicester, UK, June 4-7, 2018,  Springer Proc. Math. Stat. \textbf{330}, 31--54 (2020).

\bibitem[C3]{catanese} F. Catanese, {Orbifold classifying spaces and quotients of complex tori}, Rendiconti Sem. Mat. Univ. Pol. Torino
Vol. \textbf{82}, 1 (2024), 35--46.

\bibitem[C4]{catsns} F. Catanese, {Manifolds with trivial Chern classes I: hyperelliptic manifolds and a question by Severi}, arXiv:2206.02646v3 (2023), to appear in Ann. Sc. Norm. Super. Pisa Cl. Sci.

\bibitem[C5]{catii} F. Catanese,
{Manifolds with trivial Chern classes II: manifolds omogenous to a torus product, coframed manifolds and a question by Baldassarri}, Int. Math. Res. Not. IMRN 2025, no. 5, Paper No. rnaf041, 22 pp.




\bibitem[CD]{dim3}  F. Catanese and A. Demleitner, 
{The classification of hyperelliptic threefolds}, Groups Geom. Dyn. \textbf{14} (2020), no. 4, 1447--1454.




\bibitem[Ca1]{ca} F. Caucci, {The basepoint-freeness threshold and syzygies of abelian varieties}, Algebra Number Theory \textbf{14} (2020), no. 4, 947--960.


\bibitem[Ca2]{ca2} F. Caucci, 
{Syzygies of Kummer varieties}, Trans. Amer. Math. Soc. \textbf{377} (2024), no. 2, 1357--1370.



\bibitem[CI]{chiIy} S. Chintapalli and J. N. N.  Iyer, 
{Embedding theorems on hyperelliptic varieties}, 
Geom. Dedicata \textbf{171} (2014), 249--264.

\bibitem[CHK]{clhoko} B. Claudon, A. H\"{o}ring and J. Koll\'ar, 
{Algebraic varieties with quasi-projective universal cover}, 
J. Reine Angew. Math. \textbf{679} (2013), 207--221.

\bibitem[De]{debarre} O. Debarre, On coverings of simple abelian varieties, Bull. Soc. Math. France, \textbf{134} (2006), no. 2, 253--260.

\bibitem[De]{dem} A. Demleitner, 
{The Classification of Hyperelliptic Groups in Dimension $4$}, preprint arXiv:2211.07998, (2022).

\bibitem[EL]{el} L. Ein and R. Lazarsfeld, {Syzygies and Koszul cohomology of smooth projective varieties of arbitrary dimension}, Invent. 
Math. \textbf{111} (1993), no. 1, 51--67.


\bibitem[EL]{elbook} L. Ein and R. Lazarsfeld,
{Lectures on the Syzygies and Geometry of Algebraic Varieties}, \url{https://www.math.stonybrook.edu/robert.lazarsfeld/LSGAV.Prelim.Draft.pdf}.



\bibitem[Fu]{dim2.1} A. Fujiki, 
{Finite automorphism groups of complex tori of dimension two}, Publ. Res. Inst. Math. Sci. \textbf{24} (1988), no. 1, 1--97.

\bibitem[F]{fuconj} T. Fujita, 
{On polarized manifolds whose adjoint bundles are not semipositive}, in {\em Algebraic geometry, Sendai, 1985}, 
Adv. Stud. Pure Math. \textbf{10}, North-Holland Publishing Co., Amsterdam, 1987, 167--178.

\bibitem[GP]{gapu} F. J. Gallego and B. P. Purnaprajna, 
{Projective normality and syzygies of algebraic surfaces},
J. Reine Angew. Math. \textbf{506} (1999), 145--180.

\bibitem[GKP]{grkepe} D. Greb, S. Kebekus and T. Peternell, 
{\'Etale fundamental groups of Kawamata log terminal spaces, flat sheaves, and quotients of abelian varieties}, 
Duke Math. J. \textbf{165}  (2016), no. 10, 1965--2004.


\bibitem[Gr]{green} M. Green, {Koszul cohomology and the geometry of projective varieties}, J. Differ. Geom. \textbf{19} (1984), no. 1, 125--171.




\bibitem[GL]{grla} M. Green and R. Lazarsfeld, {On the projective normality of complete linear series on an algebraic curve}, Invent. Math., \textbf{83} (1986), no. 1, 73--90.





\bibitem[It]{ito0} A. Ito, {$M$-regularity of $\Q$-twisted sheaves and its application to linear systems on abelian varieties}, Trans. Am. Math. Soc. \textbf{375}, No. 9, 6653--6673 (2022).



\bibitem[JP]{jipa} Z. Jiang and G. Pareschi, {Cohomological rank functions on abelian varieties}, 
Ann. Sci. \'Ec. Norm. Sup\'er. (4) \textbf{53} (2020), no. 4, 815--846.


\bibitem[KO]{kaok} K. Kawatani and S. Okawa, 
 {Nonexistence of semiorthogonal decompositions and sections of the canonical bundle}, preprint arXiv:1508.00682v2 (2018).




\bibitem[Ke]{ke} G. Kempf, {Projective coordinate rings of abelian varieties}, in {\em Algebraic analysis, geometry and number theory}, Johns Hopkins Univ. Press, Baltimore, MD, 1989, 225--235.

\bibitem[KL]{kola} J. Koll\'ar and M. Larsen, 
{Quotients of Calabi-Yau varieties}, in {\em Algebra, arithmetic, and geometry: in honor of Yu. I. Manin. Vol. II}, 
Progr. Math. \textbf{270}, Birkh\"auser Boston, Boston, MA, 2009, 179--211.





\bibitem[La]{lange} H. Lange, 
{Hyperelliptic varieties}, Tohoku Math. J. (2) \textbf{53} (2001), no. 4, 491--510.

\bibitem[L1]{lasampl} R. Lazarsfeld,  {A sampling of vector bundle techniques in the study of linear series}, in {\em Lectures on Riemann surfaces}, World Sci. Publ., Teaneck, NJ, 1989, 500--559.



\bibitem[L2]{laI} R. Lazarsfeld, {\em Positivity in algebraic geometry I}, Springer-Verlag, Berlin, 2004.





\bibitem[LT]{luta} S. Lu and B. Taji, 
{A characterization of finite quotients of abelian varieties}, Int. Math. Res. Not. IMRN (2018), no. 1, 292--319.





\bibitem[Mu1]{mum2} D. Mumford, {Varieties defined by quadratic equations}, in {\em Questions on Algebraic Varieties (C.I.M.E., III Ciclo, Varenna, 1969)}, Edizioni Cremonese, Rome, 1970, 29--100. 

\bibitem[Mu2]{mum3} D. Mumford, {\em Abelian varieties}, Second edition, Oxford University Press, Oxford, 1974.

\bibitem[Na]{na}  N. Nakayama, 
{Projective algebraic varieties whose universal covering spaces are biholomorphic to $\mathbb{C}^n$}, J. Math. Soc. Japan \textbf{51} (1999), no. 3, 643--654.

\bibitem[OS]{ogsa} K. Oguiso and J. Sakurai, 
{Calabi-Yau threefolds of quotient type}, Asian J. Math. \textbf{5} (2001), no. 1, 43--77.

\bibitem[Pa]{pa1} G. Pareschi,  {Syzygies of abelian varieties}, J. Amer. Math. Soc. \textbf{13} (2000), no. 3, 651--664.

\bibitem[PP1]{ppI} G. Pareschi and M. Popa, 
{Regularity on abelian varieties, I}, J. Amer. Math. Soc. \textbf{16} (2003), no. 2, 285--302.




\bibitem[PP2]{ppIII} G. Pareschi and M. Popa, {Regularity on abelian varieties III: relationship with generic vanishing
and applications}, in {\em Grassmannians, moduli spaces and vector bundles}, Clay Math. Proc., \textbf{14}, Amer.
Math. Soc., Providence, RI, 2011, 141--167.







\bibitem[Re]{reider} I. Reider, 
{Vector bundles of rank $2$ and linear systems on algebraic surfaces},
Ann. Math. (2) \textbf{127} (1988), no. 2, 309--316.


\bibitem[SGA6]{kleiman} 
{\em Th\'eorie des intersections et th\'eor\`eme de Riemann-Roch},  
S\'eminaire de G\'eom\'etrie Alg\'ebrique du Bois-Marie 1966-1967 (SGA 6). 
Dirig\'e par P. Berthelot, A. Grothendieck et L. Illusie. Avec la collaboration de D. Ferrand, J. P. Jouanolou, O. Jussila, S. Kleiman, M. Raynaud et J. P. Serre, Lecture Notes in Math., Vol. \textbf{225}, Springer-Verlag, Berlin-New York, 1971.
 


\bibitem[Se]{serrano}  F. Serrano, 
{Divisors of bielliptic surfaces and embeddings in $\P^4$}, 
Math. Z. \textbf{203} (1990), no. 3, 527--533.

\bibitem[S-B]{sb} N. Shepherd-Barron, 
{Unstable vector bundles and linear systems on surfaces in characteristic p},
Invent. Math. \textbf{106} (1991), no. 2, 243--262.

\bibitem[Sh]{sh} T. Shibata, {The classification of smooth quotients of abelian surfaces}, preprint arXiv:2308.00884v2, (2023).


\bibitem[Ue]{ueno} K. Ueno, 
{Classification of algebraic varieties, I}, 
Compos. Math. \textbf{27} (1973), 277--342.



\bibitem[Yo]{dim2.2} H. Yoshihara, 
{Quotients of abelian surfaces}, Publ. Res. Inst. Math. Sci. \textbf{31} (1995), no. 1, 135--143.








\end{thebibliography}
\end{document}